\documentclass{article}
\usepackage{graphicx} % Required for inserting images
\usepackage{amsmath}
\usepackage{amsfonts}
\usepackage{amssymb}
\usepackage{xcolor}
\usepackage{soul}
\usepackage{amsthm}
\usepackage{hyperref}

\def\R{\mathbb{R}}
\def\N{\mathbb{N}}

\def\C{\mathbb{C}}

\def\H{\mathbb{H}}
\def\S{\mathbb{S}}
\def\V{\mathcal{V}}
\def\I{\mathcal{I}}
\def\st{\, | \,}

{\theoremstyle{plain}
\newtheorem{thm}{Theorem}[section]
\newtheorem{lemma}[thm]{Lemma}
\newtheorem{cor}[thm]{Corollary}

\newtheorem{defi}[thm]{Definition}
\newtheorem{innercustomthm}{Theorem}

}
\newenvironment{customthm}[1]
  {\renewcommand\theinnercustomthm{#1}\innercustomthm}
  {\endinnercustomthm}
  
{\theoremstyle{remark}
\newtheorem{rk}[thm]{Remark}

}

\newcommand{\fri}{{\mathrm{i}}}
\newcommand{\frj}{{\mathrm{j}}}
\newcommand{\frk}{{\mathrm{k}}}

\title{On the geometry of zero sets of central quaternionic polynomials}
\author{Gil Alon and Elad Paran}

\begin{document}
\maketitle
\begin{abstract}
Let $R$ be the ring $\H[x_1,\ldots,x_n]$ of polynomials in $n$ central variables over the real quaternion algebra $\H$, and let $I$ be a left ideal in $R$. We prove that if $p \in R$ vanishes at all the common zeros of $I$ in $\H^n$ with commuting coordinates, then as a slice regular quaternionic function, $p$ vanishes at all common zeros of $I$ in $\H^n$. This confirms a conjecture of Gori, Sarfatti and Vlacci, who settled the two dimensional case. \end{abstract}

\section{Introduction}
Let $R=\H[x_1,\dotsc,x_n]$ be the ring of polynomials in $n$ commuting variables over the algebra of quaternions.
In our previous work \cite{AP21}, we have considered the space 
$$\H^n_c=\{ (q_1,\dotsc,q_n)\in \H^n \, | \, q_i q_j=q_jq_i \text{ for all } 1\leq i,j \leq n \}$$
The points of $\H^n_c$ are a natural choice of a domain for polynomials in $R$, to be interpreted as functions. 
For one, a point $(q_1,\dotsc,q_n)$ can be plugged in a polynomial in $R$ without having to define the order in which the powers of the coordinates $q_i$ are multiplied. Moreover, we have proved (loc. cit.):
\begin{enumerate}
    \item 
    The maximal left ideals of $R$ are in one-to-one correspondence with the points of $\H^n_c$.
    \item
    The weak Nullstellensatz: For any proper left ideal $I\subset R$, there exists a point $p\in \H^n_c$ such that $f(p)=0$ for all $f\in I$.
    \item 
    The strong Nullstellensatz: For any subset $I\subseteq R$, let us define $\V_c(I)=
    \{p\in \H^n_c \, | \, f(p)=0 \text{ for all } f\in I\}$. 
    For a subset $A \subseteq \H_c^n$,  let $\I(A)=\{f \in R\, | \,f(a)=0 \text{ for all } a\in A\}$. Then for any left ideal $I \subseteq R$ we have $\I(\V_c(I))=\sqrt{I}$, where the \emph{radical} $\sqrt{I}$ is defined as the intersection of all completely prime left ideals containing $I$. 
\end{enumerate}

In a recent paper \cite{GSV}, Gori Sarfatti and Vlacci suggested considering the entire quaternionic affine space $\H^n$ as a domain for the elements of $R$. For such an evaluation to be well defined, a linear order of the variables $x_1,\dotsc, x_n$ has to be picked. 
We then get a function $p:\H^n \rightarrow \H$ which belongs to the class of \emph{slice regular functions}. The theory of such functions, which is a quaternionic multi-variable extension of complex analysis, has been developed over the last 18 years, see \cite{GSS} and \cite{GP}.

In this setting Gori et al. ask whether the strong Nullstellensatz holds when all the $\H^n$ points are considered, i.e. does $\I(\V(I))=\sqrt{I}$ hold, for any left ideal $I\subseteq R$, where $\V(I)=\{ p\in \H^n \, | \, f(p)=0 \text{ for all } f\in I \}$?

The main result of \cite{GSV} states that this equality holds when $n=2$. Gori et al. conjectured that the same result holds in general, and gave some examples where the equality holds for $n>2$ (See ibid., Example 4.17 and the end of the introduction).

One should note that in general, if $A\subseteq \H^n$ is a set, $\V(A)$ need not be a left ideal. Hence the conjecture of Gori et al. implies, non-trivially, that for a set $A=\V(I)$, $\I(A)$ is indeed a left ideal.

Given our results in \cite{AP21}, the conjecture of Gori et al. is equivalent to the equality $\I(\V(I))=\I(\V_c(I))$. In this paper, we confirm this conjecture. Indeed, we prove:

\begin{thm}\label{thm:main}
Let $I$ be a left ideal of $R=\H[x_1,\dotsc,x_n]$. If a polynomial $f\in R$ vanishes on $\V_c(I)$, then $f$ vanishes on $\V(I)$.
\end{thm}

As discussed above, from our main theorem, we immediately conclude the following strong Nullstellensatz:
\begin{cor}
Let $I$ be a left ideal of $R=\H[x_1,\dotsc,x_n]$. Then $\I(\V(I))=\sqrt{I}$.
\end{cor}

The proof of Theorem \ref{thm:main} has a geometric nature, and sheds some light on the structure of zero sets of polynomials in $R$. We define the notion of an \emph{embedded sphere} in $\H^n$ (see definition \ref{def:embedded_sphere} below). In analogy with the one-variable theory, we prove that a polynomial $f\in R$ that vanishes at two distinct points of an embedded sphere, vanishes on the entire sphere. Given a polynomial $f \in R$ that vanishes on $\V_c(I)$, and a point $p\in \V(I)$, we construct a finite set of embedded spheres, organized in the structure of a binary tree. At the root of the tree there is a sphere that contains the point $p$, and the spheres at the leaves each contain two points in $\H^n_c$. By an inductive argument from the root the leaves, we show that all the spheres are contained in $\V(I)$. In particular, each sphere at the leaves of the tree contains two points in $V_c(I)$. Therefore, $f$ vanishes on those points, and by the above claim on embedded spheres, $f$ vanishes on the entire sphere. We then turn backwards and perform an induction from the leaves to the root, to show that $f$ vanishes on all the spheres. As the sphere at the root of the tree contains $p$, we conclude that $f(p)=0$.

The rest of this paper is organized as follows: In section \ref{sec:prelim} we recall some definitions and basic facts about the quaternions and the central polynomial ring $R$. In section \ref{sec:spheres} we introduce the notion of an embedded sphere and prove some of its properties. Finally, in section \ref{sec:proof} we prove our main result. 
\section{Prelinimaries} \label{sec:prelim}

Let $\H=\R + \R \fri + \R \frj + \R \frk$ be Hamilton's algebra of quaternions\footnote{Throughout this paper, we shall denote the standard non-real generators of $\H$ by the symbols $\fri,\frj,\frk$, as opposed to the letters $i,j,k$ which we reserve for indices.}, 
and let $R=\H[x_1,\dotsc,x_n]$ be the ring of central polynomials over $\H$. An element of $R$ has the form 
$f(x)=\sum_I a_I x^I$ where $x=(x_1,\dotsc,x_n)$, the sum is over multi-indices $I=(i_1,\dotsc,i_n)\in {(\N \cup \{0\})}^n$, and $x^{(i_1,\dotsc,i_n)}=\prod_{u}{x_u}^{i_u}$.

The multiplication in $R$ is defined as follows:
$$\sum a_I x^I \sum b_J x^J=\sum_{I,J}a_Ib_Jx^{I+J}$$ 

In orded to define the substitution of a point $p=(q_1,\dotsc,q_n)\in \H^n$ in a polynomial $f(x)=\sum_I a_I x^I$, we choose the following linear order on the variables: $x_1>x_2>\dotsc>x_n$. We therefore define
$$f(p)=\sum_{I=(i_1,\dotsc,i_n)\in {(\N \cup \{0\})}^n} a_I {x_1}^{i_1}\cdot \dotsc \cdot {x_n}^{i_n}$$
With this substitution in mind, given a left ideal $I$ in $R$, we consider two notions of ``vanishing sets":

$$\V(I) = \{p \in \H^n \st f(p) = 0\hbox{ for all } f \in I\}$$
and
$$\V_c(I) = \{p \in \H_c^n \st f(p) = 0\hbox{ for all } f \in I\}.$$

Clearly, we always have $\V_c(I) \subseteq \V(I)$; Equality does not always hold, see \cite[Example 4.5]{GSV}. 

\begin{defi}
Let $A$ be an associative ring. A left ideal $P$ in $A$ is called {\it completely prime}, if for any $a,b \in P$ with $ab \in P$ and $Pb \subseteq b$ it follows that $a \in P$ or $b \in P$.  \end{defi}

This notion was introduced in \cite{Reyes} by Reyes, who showed that one-sided completely prime ideals are, from certain aspects of non-commutative ring theory, ``better" analogues of prime ideals in commutative rings than the naive notion (see also \cite{Glas}). This notion is key to the formulation of the quaternionic central Nullstellensatz in \cite{AP21}, which states:

\begin{thm}\label{central_null}
Let $I$ be a left ideal in $R$. A polynomial $f \in R$ vanishes at all points of $\V_c(I)$ if and only if $f$ belongs to the {\it radical} $\sqrt{I}$ of $I$ -- the intersection of all completely prime left ideals of $R$ containing $I$. \end{thm}

We note that a different characterization for the radical $\sqrt{I}$ was given in a recent paper of Aryapoor \cite{Ar23}, from which he obtains a more explicit version of the Nullstellensatz of \cite{AP21} . This characterization shall not be needed in the sequel. 

As we stated in the introduction, Theorem \ref{central_null} can be stated succintly in the equality $\I(\V_c(I))=\sqrt{I}$, for a left ideal $I\subseteq R$. Our main theorem in this paper (Theorem \ref{thm:main}) states that $\I(\V(I))=\I(\V_c(I))$, from which it will follow that $\I(\V(I))=\sqrt{I}$.

We now describe some well known and easy facts about the quaternions. Consider the set $\S= \{q\in \H \, | \, q^2 = -1\}$.

\begin{lemma}\label{lem:quaternion_facts} 
\begin{enumerate}
\item
$\S=\{a \fri+b \frj+c \frk\, | \,a,b,c\in \R, a^2+b^2+c^2=1\}$
\item 
Any $q\in \mathbb H$ has a representation as $q=a+bI$ where $a,b\in \mathbb {R}$ and $I\in \S$. 
\item
Any two elements of $\S$ are conjugate in the multiplicative group $\H^*$.
\end{enumerate}
\end{lemma}

Let us define, for any $I\in \S$, $L_I=\R + \R I$. Such a set is a maximal commutative subfield of $\H$, isomorphic to $\C$. By Lemma \ref{lem:quaternion_facts}, $\H = \bigcup_{I \in \S} \S_I$. The sets $L_I$ are called the \emph{slices} of $\H$ (from which the term 'slice regular functions' is derived). Given $I_1, I_2\in \S$, we either have $I_1=\pm I_2$, in which case $L_{I_1}=L_{I_2}$, or $L_{I_1} \cap L_{I_2}= \R$. Consequently, two elements of $\H$ commute if and only if they belong to the same slice, and more generally,
\begin{equation} \label{eq:slices}
\H^n_c=\bigcup_{I\in \S}(L_I)^n
\end{equation}

\begin{rk} The notion of substitution in polynomials in this paper (as well as in \cite{AP21})  differs from that in  \cite{GSV}, \cite{GP} and \cite{GSS}. We adhere to the notation $f(x)=\sum a_I x^I$ while in the above works the notation $f(x)=\sum  x^I a_I$ is used. Being in a non-commutative setting, these definitions produce different interpretations of central polynomials as functions on $\H^n$. However, as noted in \cite{GSV} it is very easy to translate between the two languages, since $\H$ is isomorphic to the opposite ring $\H^{\text{op}}$ via the conjugation map. Therefore our results hold in the opposite notation as well, when left ideals are replaced with right ideals.
\end{rk}

\section{Embedded spheres} \label{sec:spheres}
For any $q\in \H$ and $v=(q_1,\dotsc,q_n)\in \H^n$, let us denote $$v^q=(qq_1q^{-1},\dotsc,q q_n q^{-1}).$$
Let us also denote for such $v$, $\S_v=v^{\H^*}$. If $v\in \H^n_c$ then by (\ref{eq:slices}), $v\in (L_I)^n$ for some $I\in \S$. We can therefore write $v=(a_1+ b_1I,\dotsc, a_n+b_n I)$ for some $a_i, b_i\in \R$. By Lemma \ref{lem:quaternion_facts} we conclude that
\begin{align*}
 \S_v &=\{(a_1+ b_1I^q,\dotsc, a_n+b_n I^q) \st q\in \H^*\} \\
 &=\{(a_1+ b_1 J,\dotsc, a_n+b_n J) \st J\in \S\}
\end{align*}

The following lemma shows that on sets of the above form, central polynomials are affine functions:

\begin{lemma}\label{lem:affine_on_sphere}
Let $f\in \H[x_1,\dotsc, x_n]$ and let $a_1,\dotsc, a_n, b_1,\dotsc, b_n\in \R$. Then there exist $q_1,q_2 \in \H$ such that for all $I\in \S$, $f(a_1+b_1I,\dotsc,a_n+b_nI)=q_1+q_2I$.    
\end{lemma}
\begin{proof}
By linear extension, it is enough to prove the lemma for $f$ of the form $f(x_1,\dotsc,x_n)={x_1}^{i_1}\cdot \dotsc \cdot {x_n}^{i_n}$. For such $f$, the claim follows by induction on $\deg f$, and noting that for $a,b,c,d\in \R$ and $I\in \S$, $(a+bI)(c+dI)=(ac-bd)+(ad+bc)I$.
\end{proof}

\begin{defi} \label{def:embedded_sphere}
A set of the form $S=\{v_1\} \times \S_{v_2}$, where  $v_1\in \H^i$, $v_2\in \H^{n-i}_c$ and $1\leq i \leq n$, is called an embedded sphere in $\H^n$.
\end{defi}

From Lemma \ref{lem:affine_on_sphere} we conclude:
\begin{lemma} \label{lem:vanish_on_sphere} 
If a polynomial $f\in \H[x_1,\dotsc,x_n]$ vanishes on two distinct points of an embedded sphere in $\H^n$, then $f$ vanishes on the entire embedded sphere.
\end{lemma}

\begin{proof}
Let the embedded sphere be  $S=\{v_1\} \times \S_{v_2}$, and
let the two given vanishing points of $f$ be $(v_1,w)$ and $(v_1,w')$. Let us write $w=a+bI$ with $a,b\in \R^{n-i}$ and $I\in \S$. As $w$ and $w'$ are conjugate, there exists $q\in \H^*$ such that $w'=w^q$, hence $w'=a^q+(bI)^q=a+b I^q$. Note that $I\neq I^q$ (since $w\neq w'$), and $I^q\in \S$.

 Consider the function $f_{v_1}(x_{i+1},\dotsc x_n)=f(v_1,x_{i+1},\dotsc, x_n)$. Note that this function is a central polynomial in $n-i$ variables. We have $f_{v_1}(w)=f_{v_1}(w')=0$. 

By Lemma \ref{lem:affine_on_sphere}, there exist quaternions $q_1,q_2$ such that for any $J\in \S$, $f_{v_1}(a+bJ)=q_1+q_2J$. Since $f_{v_1}(a+bI)=f_{v_1}(a+bI^q)=0$, we get $q_1+q_2I=q_1+q_2 I^q=0$. Solving for $q_1,q_2$, we get $q_1=q_2=0$. Hence $f_{v_1}(a+bJ)=0$ for \emph{any} $J\in \S$. The result follows.
\end{proof}

We note that Lemma \ref{lem:vanish_on_sphere} is an $n$-variable generalization of \cite[Lemma 3.1]{GSS}.

\section{Proof of the main theorem} \label{sec:proof}

\begin{lemma} \label{lem:conj_root}
Let $I$ be a left ideal in $\H[x_1,\dotsc,x_n]$ and let $v=(q_1,\dotsc,q_n)\in \V(I)$. Let $1\leq i\leq n-1$ such that $q_i\neq 0$, and define $v_1=(q_1,\dotsc,q_i)$ and $v_2=(q_{i+1},\dotsc,q_n)$. Then $(v_1,v_2^{q_i})\in \V(I)$.    
\end{lemma}
\begin{proof} Let us first prove that the following formula holds for any $f\in \H[x_1,\dotsc x_n]$:
$$ (x_if)(v_1,v_2)=f(v_1,v_2^{q_i}) q_i$$
Indeed, it is enough to verify the formula for $f$ of the form $f={x_1}^{k_1}\dotsc {x_n}^{k_n}$. For such $f$, we have
\begin{align*}
    (x_if)(v_1,v_2)&={q_1}^{k_1}\dotsc {q_{i-1}}^{k_{i-1}}{q_i}^{k_i+1} {q_{i+1}}^{k_{i+1}} \dotsc q_n^{k_n}\\
    &={q_1}^{k_1}\dotsc {q_i}^{k_i} q_i {q_{i+1}}^{k_{i+1}} (q_{i}^{-1} q_i) \dotsc({q_i}^{-1} q_i) q_n^{k_n} ({q_i}^{-1} q_i) \\
    &={q_1}^{k_1}\dotsc {q_i}^{k_i} ({q_{i+1}}^{q_i})^{k_{i+1}} \dotsc ({q_n}^{q_i})^{k_n} q_i\\
    &=f(v_1,v_2^{q_i}) q_i
\end{align*}
Now, for any $f\in I$, we have $x_q f\in I$, hence
$ f(v_1,v_2^{q_i}) q_i=(x_if)(v_1,v_2)=0 $, so  $f(v_1,v_2^{q_i})=0$. \qedhere

\end{proof}
\begin{lemma}\label{lem:vanish_on_sphere_2}
Let $I$ be a left ideal in $\H[x_1,\dotsc,x_n]$ and let $v=(q_1,\dotsc,q_n)\in \V(I)\subseteq \H^n$. Let $1\leq i \leq n-1$ and assume that $(q_{i+1},\dotsc,q_n)\in \H^{n-i}_c$ but $(q_{i},\dotsc,q_n)\notin \H^{n-i+1}_c$. Then $\{(q_1,\dotsc,q_i)\} \times \S_{q_{i+1},\dotsc,q_n}\subseteq \V(I)$.   
\end{lemma}
\begin{proof}
Let $v_1=(q_1,\dotsc,q_i)$ and $v_2=(q_{i+1},\dotsc,q_n)$. For any $f\in I$ we have, by Lemma \ref{lem:conj_root}, $f(v_1,v_2)=f(v_1,{v_2}^q_i)=0$. By the assumptions we have $v_2\neq {v_2}^{q_i}$. By Lemma \ref{lem:vanish_on_sphere}, we conclude that $f$ vanishes on $\{v_1\}\times \S_{v_2}$. The result follows.
\end{proof}
We now prove our main theorem.

\begin{customthm}{\ref{thm:main}}
Let $I$ be a left ideal of $R=\H[x_1,\dotsc,x_n]$. If a polynomial $f\in R$ vanishes on $\V_c(I)$, then $f$ vanishes on $\V(I)$.
\end{customthm}

\begin{proof}
Let us assume that $f\in R$ vanishes on $\V_c(I)$, and let $p=(q_1,\dotsc,q_n)\in \V(I)$. If $p\in \V_c(I)$ then clearly, $f(p)=0$. Otherwise, there exists $1\leq n_1 \leq n-1$ such that $(q_{{n_1}+1},\dotsc,q_n)\in \H^{n-{n_1}}_c$ but $(q_{n_1},\dotsc,q_n)\notin \H^{n-{n_1}+1}_c$. Let $w_1=(q_{n_1+1},\dotsc,q_n)$. 
By Lemma \ref{lem:vanish_on_sphere_2}, 
\begin{align} \label{eq:sphere_inclusion}
\{(q_1,\dotsc,q_{n_1})\} \times \S_{w_1}\subseteq \V(I). 
\end{align}
Since $w_1\in \H^{n-{n_1}}_c$, there exists $I_1\in \S$ such that $w_1\in (\R+\R I_1)^{n-n_1}$. Let $I_2\in \S$ be such that $q_{n_1}\in \R + \R I_2$. Note that $q_{n_1}\notin \R$. By Lemma \ref{lem:quaternion_facts}, $I_1$ and $\pm I_2$ are conjugate. Let $\alpha^{+}_1, \alpha^{-}_1 \in \H^*$ be such that $I_1^{\alpha^{\pm}_1}=\pm I_2$. By (\ref{eq:sphere_inclusion}),
We have $(q_1,\dotsc,q_{n_1}, {w_1}^{\alpha^{\pm}_1}) \in \V(I)$.

Next, let $0\leq n_2 <n_1$ be such that $q_{n_2+1},\dotsc q_{n_1} \in \R+ \R I_2$ but either $q_{n_2} \notin \R+\R I_2$ or $n_2=0$. Let $w_2=(q_{n_2+1},\dotsc,q_{n_1})$. Note that $(w_2,w_1^{\alpha^{\pm}_1})\in (\R+\R I_2)^{n-n_2} \subset \H^{n-n_2}_c$. By Lemma \ref{lem:vanish_on_sphere_2} we have 
\begin{align} \label{eq:sphere_inclusion_2}
\{(q_1,\dotsc,q_{n_2})\} \times \S_{(w_2,w_1^{\alpha^{\pm}_1})}\subseteq \V(I). 
\end{align}
If $n_2>0$, let $I_3\in \S$ be such that $q_{n_2}\in \R+\R I_3$. Let $\alpha^{\pm}_2 \in \H^*$ be such that $I_1^{\alpha^{\pm}_2}=\pm I_3$. By (\ref{eq:sphere_inclusion_2}),
We have $(q_1,\dotsc,q_{n_2}, {w_2}^{\alpha^{\pm}_2},{w_1}^{\alpha^{\pm}_2 \alpha^{\pm}_1}) \in \V(I)$ for any choice of signs in $\alpha^{\pm}_1, \alpha^{\pm}_2$.

Continuing in this fashion, we get $n=n_0>n_1>n_2>\dotsc>n_k=0$, $I_1,\dotsc, I_{k} \in \S$ and $\alpha^{\pm}_1,\dotsc,\alpha^{\pm}_{k-1}\in \H^*$ such that, for $w_i= (q_{{n_i}+1},\dotsc,q_{n_{i-1}})$, we have:

\begin{enumerate}
\item 
For $1\leq i \leq k-1$, $I_i^{\alpha^{\pm}_i}=\pm I_{i+1}$ and $w_i^{\alpha^+_i}\neq w_i^{\alpha^-_i}$.
\item
For $1 \leq i \leq k$, $(w_i, w_{i-1}^{\alpha^{\pm}_{i-1}},\dotsc, w_1^{\alpha^{\pm}_{i-1}
\alpha^{\pm}_{i-2}\dotsc \alpha^{\pm}_1}) \in (\R+\R I_k)^{n-n_k}\subset \H^{n-n_i}_c$, for any choice of signs in $\alpha^{\pm}_1,\dotsc, \alpha^{\pm}_{i-1}$.
\item 
For $1\leq i\leq k$, $(w_k,w_{k-1},\dotsc,w_i,w_{i-1}^{\alpha^{\pm}_{i-1}},\dotsc,
w_1^{\alpha_{i-1}^{\pm} \dotsc \alpha_1^{\pm}}) \in \V(I)$ for any choice of signs in $\alpha^{\pm}_1,\dotsc, \alpha^{\pm}_{i-1}$ (Note that $(w_k,\dotsc ,w_i)=(q_n,\dotsc,q_{n_{i-1}})$).
\end{enumerate}
Now, setting $i=k$ in the above, we get that $$(w_k,w_{k-1}^{\alpha^{\pm}_{k-1}},\dotsc, w_1^{\alpha^{\pm}_{k-1}\alpha^{\pm}_{k-2}\dotsc \alpha^{\pm}_1}) \in \V_c(I)$$ for any choice of signs in $\alpha^{\pm}_1,\dotsc, \alpha^{\pm}_{k-1}$, therefore $f$ vanishes on these points. Hence, for any choice of signs in $\alpha^{\pm}_1,\dotsc,\alpha^{\pm}_{k-2}$, $f$ vanishes on two distinct points of the embedded sphere $\{w_k\}\times 
\S_{(w_{k-1},w_{k-2}^{\alpha^{\pm}_{k-2}},\dotsc, w_1^{\alpha^{\pm}_{k-2}\dotsc \alpha^{\pm}_1})}$. 

By Lemma \ref{lem:vanish_on_sphere}, $f$ vanishes on the entire embedded sphere. In particular, for any choice of signs in $\alpha^{\pm}_1,\dotsc,\alpha^{\pm}_{k-2}$, we have $$f(w_k,w_{k-1},w_{k-2}^{\alpha^{\pm}_{k-2}},\dotsc, w_1^{\alpha^{\pm}_{k-2}\dotsc \alpha^{\pm}_1})=0$$

We now get that that for any choice on signs in $\alpha^{\pm}_1,\dotsc,\alpha^{\pm}_{k-3}$, $f$ vanishes on two distinct points of the embedded sphere $\{w_k,w_{k-1},w_{k-2}\}\times 
\S_{w_{k-3}^{\alpha^{\pm}_{k-3}},\dotsc, w_1^{\alpha^{\pm}_{k-3}\dotsc \alpha^{\pm}_1}}$. 

Continuing in this fashion, we get at the last stage that $f(w_k,\dotsc,w_2,w_1^{\alpha^{\pm}_1})=0$ and conclude, in the same way, that $f(p)=f(w_k,\dotsc,w_1)=0$. \qedhere

\end{proof}

Let us conclude this paper with the following question: Does Theorem \ref{thm:main} hold if ones replaces the quaternion algebra $\H$ with an arbitrary division ring?

\end{document}